\documentclass[oneside,english]{amsart}
\usepackage[T1]{fontenc}
\usepackage[latin9]{inputenc}
\usepackage{amstext}
\usepackage{amsthm}
\usepackage{amssymb}

\makeatletter
\numberwithin{equation}{section}
\numberwithin{figure}{section}
  \theoremstyle{remark}
  \newtheorem*{acknowledgement*}{\protect\acknowledgementname}
\theoremstyle{plain}
\newtheorem{thm}{\protect\theoremname}[section]
  \theoremstyle{remark}
  \newtheorem{rem}[thm]{\protect\remarkname}
  \theoremstyle{plain}
  \newtheorem{lem}[thm]{\protect\lemmaname}
  \theoremstyle{definition}
  \newtheorem{example}[thm]{\protect\examplename}
  \theoremstyle{plain}
  \newtheorem*{conjecture*}{\protect\conjecturename}

\renewcommand{\epsilon}{\varepsilon}
\renewcommand{\phi}{\varphi}

\DeclareMathOperator{\Irr}{Irr}
\DeclareMathOperator{\Ind}{Ind}

\DeclareMathOperator{\Tr}{tr}
\DeclareMathOperator{\tr}{tr}

\DeclareMathOperator{\Ker}{Ker}
\DeclareMathOperator{\Hom}{Hom}

\DeclareMathOperator{\GL}{GL}
\DeclareMathOperator{\SL}{SL}
\DeclareMathOperator{\M}{M}
\DeclareMathOperator{\Lie}{Lie}

\DeclareMathOperator{\chara}{char} 

\newcommand{\iso}{\mathbin{\kern.15em\widetilde{\hphantom{\hspace{.6em}}}\kern-.98em\rightarrow\kern.05em}}
\newcommand{\longiso}{\mathbin{\kern.3em\widetilde{\hphantom{\hspace{1.1em}}}\kern-1.55em\longrightarrow\kern.1em}}

\newcommand{\mfg}{\mathfrak{g}}

\newcommand{\mfp}{\mathfrak{p}}

\newcommand{\mfo}{\mathfrak{o}}

\newcommand{\mfA}{\mathfrak{A}}
\newcommand{\mfP}{\mathfrak{P}}

\newcommand{\C}{\ensuremath{\mathbb{C}}}
\newcommand{\F}{\ensuremath{\mathbb{F}}}

\newcommand{\Q}{\ensuremath{\mathbb{Q}}}

\newcommand{\Z}{\ensuremath{\mathbb{Z}}}

\usepackage{tikz-cd}

\usepackage[shortlabels]{enumitem}
\setenumerate{label=(\roman*)}

\newcommand{\Amin}{\mfA_{\mathrm{m}}} 
\newcommand{\Pmin}{\mfP_{\mathrm{m}}}
\newcommand{\Umin}{U_{\mathrm{m}}}

\newcommand{\Amax}{\mfA_{\mathrm{M}}} 
\newcommand{\Pmax}{\mfP_{\mathrm{M}}}
\newcommand{\Umax}{U_{\mathrm{M}}}

\newcommand{\Hmin}{H_{\mathrm{m}}}
\newcommand{\Jmin}{J_{\mathrm{m}}}
\newcommand{\Hmax}{H_{\mathrm{M}}}
\newcommand{\Jmax}{J_{\mathrm{M}}}

\newcommand{\JmM}{J_{\mathrm{m,M}}}

\newcommand{\thetam}{\theta_{\mathrm{m}}}
\newcommand{\thetaM}{\theta_{\mathrm{M}}}
\newcommand{\etam}{\eta_{\mathrm{m}}}
\newcommand{\etaM}{\eta_{\mathrm{M}}}

\newcommand{\hatetaM}{\hat{\eta}_{\mathrm{M}}}

\newcommand{\emin}{e_{\mathrm{m}}}
\newcommand{\emax}{e_{\mathrm{M}}}

\newcommand{\mfr}{\mathfrak{r}}
\newcommand{\mfj}{\mathfrak{j}}

\makeatother

\usepackage{babel}
  \providecommand{\acknowledgementname}{Acknowledgement}
  \providecommand{\conjecturename}{Conjecture}
  \providecommand{\examplename}{Example}
  \providecommand{\lemmaname}{Lemma}
  \providecommand{\remarkname}{Remark}
\providecommand{\theoremname}{Theorem}

\begin{document}
\title[Representations of $\GL_N$ - an overview]{Representations of $\GL_N$ over finite local principal ideal rings - an overview}

\author{Alexander Stasinski}
\begin{abstract}
We give a survey of the representation theory of $\GL_{N}$ over finite
local principal ideal rings via Clifford theory, with an emphasis
on the construction of regular representations. We review results
of Shintani and Hill, and the generalisation of Takase. We then summarise
the main features, with some details but without proofs, of the recent
constructions of regular representations due to Krakovski\textendash Onn\textendash Singla
and Stasinski\textendash Stevens, respectively.
\end{abstract}

\address{Department of Mathematical Sciences, Durham University, South Rd,
Durham, DH1 3LE, UK}

\email{alexander.stasinski@durham.ac.uk}
\maketitle

\section{Introduction}

This paper is a survey of the (complex) representation theory of the
group $\GL_{N}(\mfo)$, where $\mfo$ is a compact discrete valuation
ring, or equivalently, the ring of integers in a non-Archimedean local
field with finite residue field $\F_{q}$ of characteristic $p$.
Since $\GL_{N}(\mfo)$ is a profinite group, we consider its continuous
representations, and a representation is continuous if and only if
it is smooth if and only if it factors through a finite quotient $\GL_{N}(\mfo_{r})$,
where $\mfo_{r}:=\mfo/\mfp^{r}$, $\mfp$ is the maximal ideal in
$\mfo$, and $r\geq1$. We therefore focus on the representations
of the finite groups $\GL_{N}(\mfo_{r})$.

The representation theory of $\GL_{N}(\mfo_{r})$ has a relatively
long history (see the historical notes in Section~\ref{sec:Historical-overview}),
and has very recently seen intensified activity from several directions.
We will focus mostly on the recent developments regarding so-called
\emph{regular} representations, studied via Clifford theory. Regular
representations roughly correspond to regular conjugacy classes of
matrices in the Lie algebra $\mfg_{r}=\M_{N}(\mfo_{r})$, that is,
matrices whose centralisers mod $\mfp$ have dimension $N$. The first
construction of this kind goes back to Shintani \cite{Shintani-68},
who constructed all the regular representations when $r$ is even.
This was followed by work of Hill \cite{Hill_regular}, who rediscovered
Shintani's construction and also provided a partial construction of
so-called split regular representations for $r$ odd. As we will see
in subsequent sections, the representation theory of $\GL_{N}(\mfo_{r})$
is much harder when $r$ is odd compared to when $r$ is even. Very
recently it was realised by Takase \cite{Takase-16} that Hill's construction
does not actually produce all the split regular representations. Furthermore,
Takase gave a construction of all regular representations which correspond
to conjugacy classes with separable characteristic polynomial mod
$\mfp$, assuming the residue characteristic $p$ of $\mfo$ is not
$2$. At the same time, and independently, two general constructions
of regular representations have been found. One is by Krakovski, Onn
and Singla \cite{KOS}, which works whenever $p$ is not $2$, and
the other is by Stasinski and Stevens \cite{SS/2016}. The latter
works for any $\mfo$, and we therefore now have a complete construction
of all the regular representations of $\GL_{N}(\mfo_{r})$. This is
currently the most general uniform construction of irreducible representations
of $\GL_{N}(\mfo_{r})$ available.

In Section~\ref{sec:Clifford-theory} we give an introduction to
the Clifford theory approach to the representations of $\GL_{N}(\mfo_{r})$.
In Section~\ref{sec:Reg-reps} we define regular representations
and give the construction when $r$ is even. In Sections~\ref{sec:Hill-Takase}-\ref{sec:SS}
we then focus on the various constructions of regular representations
for $r$ odd. Section~\ref{sec:Hill-Takase} contains a summary of
Hill's and Takase's constructions of regular semisimple representations.
In Section~\ref{sec:KOS} we give an outline of the construction
of Krakovski, Onn and Singla. Finally, in Section~\ref{sec:SS} we
elaborate on the main steps in the construction of Stasinski and Stevens.
In the final Section~\ref{sec:Open-problems} we mention some open
problems.

Throughout the paper we have omitted most proofs, apart from the proofs
of some occasional lemmas. On the other hand, we have tried to provide
detailed explanations of many of the arguments.
\begin{acknowledgement*}
I wish to thank the organisers of the conference ``Around Langlands
correspondences'', at the Universit\'e Paris-Sud in June 2015 for
the opportunity to talk about the work \cite{SS/2016}. I am grateful
to Shaun Stevens for many helpful conversations, and to Uri Onn for
explaining some details of \cite{KOS}.
\end{acknowledgement*}

\section{Historical overview\label{sec:Historical-overview}}

The characters of $\GL_{N}(\mfo_{1})=\GL_{N}(\F_{q})$ were determined
in a classical paper of Green \cite{Green} and the representations
can be constructed via Deligne-Lusztig theory. The representations
of the finite groups $\mbox{GL}_{2}(\mfo_{r})$, for all $r\geq1$,
have in one form or another been known to some mathematicians since
the late 70s. There have been at least two different approaches to
this problem. On the one hand, there is the Weil representation approach
of Nobs and Wolfart (applied to the case $\mathfrak{o}=\mathbb{Z}_{p}$
in \cite{Nobs-GL2}). On the other hand, there is the approach via
orbits and Clifford theory due to Kutzko (unpublished), and independently
to Nagornyj \cite{nagornyj1}; see also \cite{Alex_smooth_reps_GL2}.

The related case of $\mathrm{SL}_{2}(\mathbb{Z}_{p})$, $p\neq2$,
was studied by Kloosterman \cite{Kloosterman_I,Kloosterman_II}, Tanaka
\cite{Tanaka1,Tanaka2}, Kutzko (thesis; unpublished), and Shalika
(for general $\mathfrak{o}$ and $p\neq2$) \cite{Shalika}. The Clifford
theoretic approach of Kutzko and Shalika was rediscovered by Jaikin-Zapirain
in \cite[Section~7]{Jaikin-zeta}. Another description of the representations
of $\mathrm{SL}_{2}(\mathbb{Z}_{p})$ (including the much more difficult
case $p=2$) was obtained by Nobs and Wolfart \cite{MR0444787,MR0444788}
using Weil representations. The case $\mathrm{PGL}_{2}(\mathfrak{o})$,
again with $p\neq2$, was treated by Silberger \cite{Silberger}. 

The representations of $\GL_{3}(\mathfrak{o})$ were studied by Nagornyj
in \cite{Nagornyj2}, but the construction of representations was
left incomplete. However, it was shown in \cite{Nagornyj2} that the
classification of representations of $\GL_{N}(\mathfrak{o})$ is a
so-called wild problem, and in general one can therefore not expect
an explicit and surveyable parametrisation of all the representations. 

Recently, thorough and in-depth work on the representations of $\GL_{3}(\mfo)$
and $\SL_{3}(\mathfrak{o})$ (and related groups) has appeared in
a series of papers by Avni, Klopsch, Onn and Voll; see, for example,
\cite{AKOV-p-adic--arithm-grps,AKOV-similarityclasses-A2}. The results
in \cite{AKOV-p-adic--arithm-grps} are based, among other things,
on the Kirillov orbit method, which works for principal congruence
subgroups of $\SL_{3}(\mathfrak{o})$ of index large enough compared
to $p$, and only when $\mfo$ has characteristic $0$. In \cite{AKOV-similarityclasses-A2}
the authors employ Clifford theoretic methods to count the representations
of $\GL_{3}(\mfo_{r})$ and $\SL_{3}(\mfo_{r})$ (when $\chara\mfo=0$
or $\chara\mfo=p$ and $p$ is large enough relative to $r$) and
of $\SL_{3}(\mathfrak{o})$ (when $\chara\mfo=0$ and $p$ is large
enough relative to the absolute ramification index of $\mfo$). Analogous
results are obtained for unitary groups corresponding to an unramified
extension of $\mfo$.

For the groups $\GL_{N}(\mfo_{r})$ with $N\geq2$, $r\geq2$, the
first general results seem to be due to Shintani in 1968 \cite{Shintani-68},
who constructed the so-called \emph{regular} representations when
$r$ is even. This construction was rediscovered in independent work
of Hill around 1995 \cite{Hill_regular}. The series of papers by
Hill \cite{Hill_Jordan,Hill_nilpotent,Hill_regular,Hill_semisimple_cuspidal}
use the method of orbits and Clifford theory to study and construct
some of the representations of $\GL_{N}(\mfo_{r})$. In particular,
in addition to Shintani's construction of regular representations
for $r$ even, Hill went on to construct certain regular representations
when $r$ is odd. Over twenty years after Hill's work, it was realised
by Takase \cite{Takase-16} that Hill's construction of split regular
representations is not exhaustive when the orbit is not semisimple.
As mentioned in the introduction, recent constructions of regular
representations have led to successively more general results, so
that we now have a complete construction of all the regular representations
of $\GL_{N}(\mfo_{r})$. We will give a more detailed description
of this work in subsequent sections.

Another approach to the representation theory of $\GL_{N}(\mfo_{r})$
is based on viewing this group as the automorphism group of a rank
$N$ $\mfo$-module. This was initiated by Onn \cite{Uri-rank-2},
who defined a new type of induction functor, called infinitesimal
induction, for general automorphism groups of $\mathfrak{o}$-modules
of residual rank $N$. Infinitesimal induction complements the classical
induction from parabolic subgroups, which in \cite{Uri-rank-2} is
referred to as geometric induction\emph{.} Decomposing these induced
representations and using the known construction of (strongly) cuspidal
representations for $\mbox{GL}_{2}(\mathfrak{o})$, leads to another
classification of the representations of this group.

Finally, we mention a different approach to the representations of
$\GL_{N}(\mfo_{r})$, or more generally, for reductive groups over
$\mfo_{r}$. This approach is a cohomological construction of certain
irreducible representations attached to characters of finite maximal
tori. It was given by Lusztig in \cite{Lusztig-Fin-Rings} in the
case where $\mfo$ has positive characteristic, and for arbitrary
$\mfo$ in \cite{Alex_Unramified_reps}. This is a higher level generalisation
of the classical construction of Deligne and Lusztig \cite{delignelusztig},
which corresponds to the case $r=1$. Another, ``purely algebraic''
(non-cohomological) construction of representations of certain split
reductive groups, also attached to characters of finite maximal tori,
was given by G\'erardin \cite{Gerardin}. In \cite[Section~1]{Lusztig-Fin-Rings}
Lusztig suggested the problem of whether these representations are
in fact the same as those given by the higher Deligne-Lusztig construction.
This was recently answered in the affirmative for $r$ even by Chen
and Stasinski \cite{Chen-Stasinski}. 

\section{Clifford theory for $\GL_{N}(\mfo_{r})$\label{sec:Clifford-theory}}

If $G$ is a finite group, we will write $\Irr(G)$ for the set of
isomorphism classes of complex irreducible representations of $G$.
For convenience, we will always consider an element $\rho\in\Irr(G)$
as a representation, rather than an equivalence class of representations,
that is, we identify $\rho\in\Irr(G)$ with any representative of
the isomorphism class $\rho$. One can view $\Irr(G)$ as the set
of irreducible characters of $G$, but we prefer to work with representations
when possible. If $G$ is abelian, we will often refer to a one-dimensional
representation of $G$ as a character. If $H\subseteq G$ is a subgroup
and $\rho$ is any representation of $G$ we write $\rho|_{H}$ for
the restriction of $\rho$ to $H$.

Let $\mathfrak{o}$ be a compact discrete valuation ring, that is,
the ring of integers in a non-Archimedean local field with finite
residue field, say $\F_{q}$, of characteristic $p$. Denote by $\mathfrak{p}$
the maximal ideal of $\mathfrak{o}$, and by $\varpi$ a fixed generator
of $\mfp$. For any integer $r\geq1$ we write $\mathfrak{o}_{r}$
for the finite ring $\mathfrak{o}/\mathfrak{p}^{r}$. We will also
use $\mfp$ and $\varpi$ to denote the corresponding images of $\mfp$
and $\varpi$ in $\mathfrak{o}_{r}$. Fix an integer $N\geq2$ and,
for any $r\geq1$, put 
\begin{align*}
G_{r} & =\GL_{N}(\mfo_{r}),\\
\mfg_{r} & =\M_{N}(\mfo_{r}),
\end{align*}
where, for a commutative ring $R$, we use $\M_{N}(R)$ to denote
the algebra of $N\times N$ matrices over $R$. From now on, we consider
a fixed $r\geq2$. For any integer $i$ such that $r\geq i\geq1$,
let $\rho_{i}=\rho_{r,i}:G_{r}\rightarrow G_{i}$ be the surjective
homomorphism induced by the canonical map $\mathfrak{o}_{r}\rightarrow\mathfrak{o}_{i}$,
and write $K^{i}=K_{r}^{i}=\Ker\rho_{i}$. We also write $\rho_{i}$
for the corresponding homomorphism $\mfg_{r}\rightarrow\mfg_{i}$.
We thus have a descending chain of subgroups
\[
G_{r}\supset K^{1}\supset\dots\supset K^{r}=\{1\},
\]
where
\[
K^{i}=1+\mfp^{i}\mfg_{r}.
\]
With this description of the kernels, it is easy to show the commutator
relation $[K^{i},K^{j}]\subseteq K^{\min(i+j,r)}$, for $r\geq i,j\geq1$.
In particular, if $i\geq r/2$, then $K_{i}$ is abelian, and if we
let $l=\lceil\frac{r}{2}\rceil$, then $K^{l}$ is the maximal abelian
group among the kernels $K^{i}$. From now on, let $i\geq r/2$, that
is, $i\geq l$. Then the map $x\mapsto1+\varpi^{i}x$ induces an isomorphism
\begin{equation}
\mfg_{r-i}\longiso K^{i}.\label{eq:Mn-Ki-iso}
\end{equation}
The group $G_{r}$ acts on $\mfg_{r-i}$ by conjugation, via its quotient
$G_{r-i}$. This action is transformed by the above isomorphism into
the action of $G_{r}$ on its normal subgroup $K^{i}$. Let $F$ be
the fraction field of $\mfo$. Fix an additive character $\psi:F\rightarrow\mathbb{C}^{\times}$
which is trivial on $\mfo$ but not on $\mfp^{-1}$. For each $r\geq1$
we can view $\psi$ as a character of the group $F/\mfp^{r}$ whose
kernel contains $\mfo_{r}$. We will use $\psi$ and the trace form
$(x,y)\mapsto\tr(xy)$ on $\mfg_{r}$ to set up a duality between
the groups $\Irr(K^{i})$ and $\mfg_{r-i}$. For $\beta\in\M_{N}(\mathfrak{o}_{r})$,
define a homomorphism $\psi_{\beta}:K^{i}\rightarrow\mathbb{C}^{\times}$
by
\begin{equation}
\psi_{\beta}(1+x)=\psi(\varpi^{-r}\tr(\beta x)),\label{eq:character-psi-b}
\end{equation}
for $x\in\mfp^{i}\mfg_{r}$. Note that $\varpi^{-r}\tr(\beta x)$
is a well defined element of $F/\mfp^{r}$. Since $\psi$ is trivial
on $\mfo_{r}$, $\psi_{\beta}$ only depends on $x\mod\mfp^{r-i}$
(as it must in order to be well defined). Moreover, the map $\beta\mapsto\psi_{\beta}$
is a homomorphism whose kernel is $\mfp^{r-i}\mfg_{r}$, thanks to
the non-degeneracy of the trace form. Hence it induces an isomorphism
\[
\mfg_{r}/\mfp^{r-i}\mfg_{r}\longiso\Irr(K^{i}),
\]
where we will usually identify $\mfg_{r}/\mfp^{r-i}\mfg_{r}$ with
$\mfg_{r-i}$. For $g\in G_{r}$ we have 
\[
\psi_{g\beta g^{-1}}(x)=\psi(\varpi^{i-r}\tr(g\beta g^{-1}x))=\psi(\varpi^{i-r}\tr(\beta g^{-1}xg))=\psi_{\beta}(g^{-1}xg).
\]
Thus the isomorphism (\ref{eq:Mn-Ki-iso}) transforms the action of
$G_{r}$ on $\mfg_{r}$ into (the inverse) conjugation of characters.
\begin{rem}
In the above we have used adjoint orbits (i.e., conjugacy classes)
in $\mfg_{r}/\mfp^{r-i}\mfg_{r}$ to parametrise orbits of characters
of $K^{i}$. From some points of view it is more natural to use co-adjoint
orbits in the dual
\[
\mfg_{r}^{*}:=\Hom_{\mfo_{r}}(\mfg_{r},\mfo_{r}).
\]
Indeed the pairing $\langle\,\cdot\,,\,\cdot\,\rangle:\mfg_{r}^{*}\times\mfg_{r}\rightarrow\mfo_{r}$
given by $\langle f,\beta\rangle=f(\beta)$ is non-degenerate and
one can define
\[
\psi_{f}(1+x)=\psi(\varpi^{-r}\langle f,x\rangle),
\]
where $f\in\mfg_{r}^{*}$. This induces an isomorphism $\mfg_{r}^{*}/\mfp^{r-i}\mfg_{r}^{*}\cong\mfg_{r-i}^{*}\iso\Irr(K^{i})$,
and has the advantage of generalising to Chevalley groups other than
$\GL_{N}$ (where the trace form may be degenerate); see \cite{Gerardin}.
However, for $\GL_{N}$ we prefer to work with elements in $\mfg_{r}$
rather than elements in its dual, and we can translate between the
two by means of the $G_{r}$-equivariant bijection induced by the
trace form.
\end{rem}
If $G$ is a finite group, $H\subseteq G$ is a subgroup and $\rho\in\Irr(H)$,
we will write $\Irr(G\mid\rho)$ for the set of $\pi\in\Irr(G)$ such
that $\pi$ contains $\rho$ on restriction to $H$, that is,
\[
\Irr(G\mid\rho)=\{\pi\in\Irr(G)\mid\langle\pi|_{H},\rho\rangle\neq0\}.
\]
Moreover, if $N$ is a normal subgroup of $G$, then $G$ acts on
$\Irr(N)$ by $\rho\mapsto{}^{g}\rho$, where $^{g}\rho(n):=\rho(gng^{-1})$,
for $g\in G$, $n\in N$. In this case, we define the \emph{stabiliser}
of $\rho\in\Irr(N)$ to be $G(\rho)=\{g\in G\mid{}^{g}\rho\cong\rho\}$.
We will subsequently make use of the following well known results
from Clifford theory of finite groups:
\begin{thm}
\label{thm:Clifford}Let $G$ be a finite group, and $N$ a normal
subgroup. Then the following hold:

\begin{enumerate}
\item \label{Clifford1}(Clifford's theorem) If $\pi\in\Irr(G)$, then $\pi|_{N}=e\bigoplus_{\rho\in\Omega}\rho$,
where $\Omega\subseteq\Irr(N)$ is an orbit under the action of $G$
on $\Irr(N)$ by conjugation, and $e$ is a positive integer.
\item \label{Clifford2}Suppose that $\rho\in\Irr(N)$. Then $\theta\mapsto\Ind_{G(\rho)}^{G}\theta$
is a bijection from $\Irr(G(\rho)\mid\rho)$ to $\Irr(G\mid\rho)$.
\item \label{Clifford3}Let $H$ be a subgroup of $G$ containing $N$,
and suppose that $\rho\in\Irr(N)$ has an extension $\tilde{\rho}$
to $H$ (i.e., $\tilde{\rho}|_{N}=\rho$). Then 
\[
\Ind_{N}^{H}\rho=\bigoplus_{\chi\in\Irr(H/N)}\tilde{\rho}\chi,
\]
where each $\tilde{\rho}\chi$ is irreducible, and where we have identified
$\Irr(H/N)$ with $\{\chi\in\Irr(H)\mid\chi(N)=1\}$.
\end{enumerate}
\end{thm}
For proofs of the above, see for example \cite{Isaacs}, 6.2, 6.11
and 6.17, respectively. The above results \ref{Clifford1} and \ref{Clifford2}
show that in order to obtain a classification of the representations
of $G_{r}$, it is enough to classify the orbits of characters $\psi_{\beta}$
of a normal subgroup $K^{i}$, and to construct all the elements in
$\Irr(G_{r}(\psi_{\beta})\mid\psi_{\beta})$, that is, to decompose
$\Ind_{K^{i}}^{G_{r}(\psi_{\beta})}\psi_{\beta}$ into irreducible
representations. This is what we shall do in the following, taking
$K^{i}=K^{l}$.
\begin{rem}
\label{rem:construction}By an (algebraic) construction of some irreducible
representations (or characters) of $G_{r}$ via Clifford theory, we
will always mean a general (i.e., valid for all $G_{r}$) finite sequence
of extensions and inductions of characters, starting from the one-dimensional
characters of $K^{l}$. Note that the existence of an extension of
a representations is allowed to be a non-constructive fact. 

In order to have a complete understanding of representations constructed
via Clifford theory, it is necessary to have an understanding of the
$G_{r}$ conjugacy classes (or orbits) in $\mfg_{r}$, because $\Irr(G_{r}(\psi_{\beta})\mid\psi_{\beta})=\Irr(G_{r}(\psi_{\beta'})\mid\psi_{\beta'})$
if $\beta$ and $\beta'$ are conjugate. One cannot expect to have
an explicit understanding of all the orbits, but we \emph{do} have
an explicit normal form for regular orbits, as we will see next.
\end{rem}

\section{\label{sec:Reg-reps}Regular representations, $r$ even}

An irreducible representation $\pi$ of $G_{r}$ is called \emph{regular
}if $\pi|_{K^{l}}$ contains $\psi_{\beta}$ with $\beta\in\mfg_{r}$
regular. By a result of Hill \cite[Theorem~3.6]{Hill_regular} $\beta\in\mfg_{r}$
is regular if and only if its image $\bar{\beta}\in\mfg_{1}=\M_{N}(\F_{q})$
is regular, that is, if $\dim C_{\mfg_{1}}(\bar{\beta})=N$. There
are several equivalent characterisations of regular elements in $\mfg_{1}$;
in particular, $\bar{\beta}\in\mfg_{1}$ is regular iff $C_{\mfg_{1}}(\bar{\beta})$
is abelian iff the characteristic polynomial of $\bar{\beta}$ equals
the minimal polynomial iff $\bar{\beta}$ is conjugate to a companion
matrix. Note that $\beta$ depends on the choice of $\psi$, but for
any other choice $\psi'$ we have $\psi_{\beta}=\psi'_{a\beta}$,
for some $a\in\mfo_{r}^{\times}$, and since $\beta$ is regular if
and only if $a\beta$ is regular, regularity is an intrinsic property
of a representation $\pi\in\Irr(G_{r})$.

There are three special properties of regular elements which will
allow us to construct and completely classify all the regular representations:
\begin{enumerate}
\item We can tell explicitly when two regular elements are $G_{r}$-conjugate,
namely, if and only if their companion matrices coincide.
\item The centraliser $C_{G_{r}}(\beta)$ of a regular element $\beta\in\mfg_{r}$
is abelian.
\item For any $1\leq s\leq r$, the map $\rho_{s}:C_{G_{r}}(\beta)\rightarrow C_{G_{s}}(\beta_{s})$
is surjective, where $\beta_{s}$ is the image of $\beta$ under $\rho_{s}:\mfg\rightarrow\mfg_{s}$.
\end{enumerate}
We will illustrate this in the construction of all regular representations
of $G_{r}$ when $r$ is even, given below.
\begin{rem}
If $\pi\in\Irr(G_{r}\mid\psi_{\beta})$, then (\ref{eq:character-psi-b})
implies that $\pi$ has $K^{r-1}$ in its kernel if and only if $\bar{\beta}=0$.
Thus $\pi$ factors through $G_{r-1}$ if and only if $\bar{\beta}=0$.
If this is the case, $\pi$ is called \emph{imprimitive}. If $\pi$
does not factor through $G_{r-1}$ it is called \emph{primitive}.
Note that a regular representation is necessarily primitive. On the
other hand, there exist irreducible representations of $G_{r}$ which
are not regular, because they factor through $G_{r-1}$, but are regular
when viewed as characters of $G_{r-1}$. For example, take the representations
of $\GL_{2}(\mfo_{4})$ with $\beta=\left(\begin{smallmatrix}0 & \pi\\
0 & 0
\end{smallmatrix}\right)$.
\end{rem}
From now on, let $\psi_{\beta}\in\Irr(K^{l})$ with $\beta\in\mfg_{r}$
regular. Let $l'=r-l$, so that $l=l'$ when $r$ is even and $l'=l-1$
when $r$ is odd. As indicated in the previous section, the stabiliser
$G_{r}(\psi_{\beta})$ plays an important role in the construction
of representations of $G_{r}$. The formula $\psi_{\beta}(g^{-1}xg)=\psi_{g\beta g^{-1}}(x)$,
together with the fact that $\psi_{\beta}=\psi_{\beta'}\Leftrightarrow\beta\equiv\beta'\mod\mfp^{l'}$,
implies that
\begin{equation}
G_{r}(\psi_{\beta})=C_{G_{r}}(\beta+\mfp^{l'}\mfg_{r}).\label{eq:Stabliser-first}
\end{equation}
An important corollary of \cite[Theorem~3.6]{Hill_regular} is that
for regular $\beta$, and any $s$ such that $r\geq s\geq1$, the
natural reduction map
\[
C_{G_{r}}(\beta)\longrightarrow C_{G_{s}}(\beta_{s})
\]
 is surjective. Another corollary of \cite[Theorem~3.6]{Hill_regular}
is that for regular $\beta$ we have $C_{G_{r}}(\beta)=\mfo_{r}[\beta]^{\times}$,
so that in particular, the centraliser is abelian. Together with (\ref{eq:Stabliser-first})
these two results imply that 
\begin{equation}
G_{r}(\psi_{\beta})=C_{G_{r}}(\beta)K^{l'}=\mfo_{r}[\beta]^{\times}K^{l'}.\label{eq:Stabiliser}
\end{equation}

We now give the construction of regular representations of $G_{r}$
in the case when $r$ is even. Suppose that $r$ is even, so that
$l=l'$. Let $\theta\in\Irr(C_{G_{r}}(\beta))$ be any irreducible
component of $\Ind_{C_{G_{r}}(\beta)\cap K^{l}}^{C_{G_{r}}(\beta)}(\psi_{\beta}|_{C_{G_{r}}(\beta)\cap K^{l}})$.
Since $C_{G_{r}}(\beta)$ is abelian $\theta$ is one-dimensional,
and hence it agrees with $\psi_{\beta}$ on $C_{G_{r}}(\beta)\cap K^{l}$.
It is then easy to check that
\[
\tilde{\psi}_{\beta}(ck):=\theta(c)\psi_{\beta}(k)
\]
is a well defined one-dimensional representation of $G_{r}(\psi_{\beta})$,
and by construction it is an extension of $\psi_{\beta}$. By Theorem~\ref{thm:Clifford}\,\ref{Clifford3}
we obtain 
\[
\Irr(G_{r}(\psi_{\beta})\mid\psi_{\beta})=\{\tilde{\psi}_{\beta}\chi\mid\chi\in\Irr(C_{G_{l}}(\beta_{l}))\},
\]
where $\beta_{l}\in\mfg_{l}$ is the image of $\beta$. Hence Theorem~\ref{thm:Clifford}\,\ref{Clifford2}
implies that there is a bijection 
\begin{align*}
\Irr(C_{G_{l}}(\beta_{l})) & \longrightarrow\Irr(G_{r}\mid\psi_{\beta})\\
\chi & \longmapsto\Ind_{G_{r}(\psi_{\beta})}^{G_{r}}\tilde{\psi}_{\beta}\chi.
\end{align*}
Note that this is not canonical, but depends on the choice of $\tilde{\psi}_{\beta}$.
We have thus constructed the irreducible representations of $G_{r}$
containing $\psi_{\beta}$, in terms of the irreducible representations
of the abelian group $C_{G_{l}}(\beta_{l})$ (which we consider known;
cf.~Remark~\ref{rem:construction}). Note that if we start with
another element in the conjugacy class of $\beta$, we obtain the
same set of irreducible representations of $G_{r}$. Thus, when $r$
is even, running through a set of representatives for the regular
conjugacy classes in $\mfg_{l}$, yields all the regular representations
of $G_{r}$ exactly once.

As far as the author is aware, the above construction is due to Shintani
\cite[\S{}2, Theorem~2]{Shintani-68}, although Shintani does not
prove that every regular element in $\mfg_{l}$ is regular mod $\mfp$.
The construction was rediscovered by Hill \cite[Theorem~4.1]{Hill_regular}.\\

It remains to construct the regular representations of $G_{r}$ when
$r$ is odd. This requires additional methods, due to the fact that
$G_{r}(\psi_{\beta})=C_{G_{r}}(\beta)K^{l'}$, and it is not possible
to extend $\psi_{\beta}$ from $K^{l}$ to $G_{r}(\psi_{\beta})$.
Instead, one has to take several intermediate steps consisting of
extensions and inductions. In the following we will give an exposition
of the currently known constructions (sometimes partial) of regular
representations of $G_{r}$, for $r$ odd.

\section{The constructions of Hill and Takase\label{sec:Hill-Takase}}

From now on and until the end of Section~\ref{sec:SS} we will assume
that $r$ is odd, so that $l':=r-l=l-1$. In this case, Hill \cite{Hill_regular}
claimed to give a construction of so-called \emph{split} regular representations,
that is, those for which the characteristic polynomial of $\bar{\beta}\in\mfg_{1}$
splits into linear factors over $\F_{q}$. Takase \cite{Takase-16}
recently pointed out a gap in the proof of Hill's result \cite[Theorem~4.6]{Hill_regular}
and proved that the construction exhausts at most the split regular
\emph{semisimple} representations, but does not exhaust all split
regular representations. We give a summary of Hill's construction
following \cite{Hill_regular}, point out two problems in the proof,
and state the correction/generalisation due to Takase.

We have an isomorphism $K^{l'}/K^{l}\cong\mfg_{1}$, and we can identify
any subgroup of $K^{l'}$ which contains $K^{l}$ with a sub-vectorspace
of $\mfg_{1}$. Define the alternating bilinear form 
\begin{align*}
B_{\beta}:K^{l'}/K^{l}\times K^{l'}/K^{l} & \longrightarrow\F_{q}\\
B_{\beta}((1+\pi^{l'}x)K^{l},(1+\pi^{l'}y)K^{l}) & =\tr(\bar{\beta}(\bar{x}\bar{y}-\bar{y}\bar{x})),
\end{align*}
where the bars denote reductions mod $\mfp$. The following is \cite[Lemma~4.5]{Hill_regular},
rewritten in our notation.
\begin{lem}
\label{lem:Hill-H}Suppose that $\beta\in\mfg_{r}$ is split regular.
Then there exists a subgroup $H_{\beta}$ of $K^{l'}$ such that $H_{\beta}$
contains $K^{l}$ and such that $H_{\beta}/K^{l}$ is a maximal isotropic
subspace of $K^{l'}/K^{l}$ with respect to the form $B_{\beta}$.
Moreover, $H_{\beta}$ is a normal subgroup of $G_{r}(\psi_{\beta})$.
\end{lem}
We recall that a subspace $U$ of a vector space $V$ with a bilinear
form $B(\,\cdot\,,\,\cdot\,)$ is called \emph{isotropic} (or sometimes
\emph{totally isotropic}) if $U\subseteq U^{\perp}$, that is, if
$B(U,U)=0$. Furthermore, $U$ is called \emph{maximal isotropic }(or
sometimes \emph{Lagrangian})\emph{ }if it is not properly contained
in any isotropic subspace, or equivalently, if $U=U^{\perp}$.

The proof of the above lemma consists of taking $H_{\beta}=(B\cap K^{l'})K^{l}$,
where $B$ is the upper-triangular subgroup of $G_{r}$, and showing
that it has the required properties, using the assumption that $\bar{\beta}$
is upper-triangular. Thus in Hill's construction, $H_{\beta}$ is
in fact independent of $\beta$.

Hill's main theorem \cite[Theorem~4.6]{Hill_regular} regarding the
construction of split regular representations for $r$ odd claims
that if $\beta\in\mfg_{r}$ is split regular, then for every $\pi\in\Irr(G_{r}\mid\psi_{\beta})$,
there exists a subgroup $H_{\beta}$ as in Lemma~\ref{lem:Hill-H}
and an extension $\tilde{\psi}_{\beta}$ of $\psi_{\beta}$ to $C_{G_{r}}(\beta)H_{\beta}$
such that 
\[
\pi=\Ind_{C_{G_{r}}(\beta)H_{\beta}}^{G_{r}}\tilde{\psi}_{\beta}.
\]
Unfortunately, Hill's proof of \cite[Theorem~4.6]{Hill_regular} suffers
from two problems. One is that a certain counting argument only goes
through when $\bar{\beta}$ is assumed to be semisimple (see \cite[Proposition~2.1.1]{Takase-16}),
so that Hill's construction does not exhaust the split regular representations.
The other problem is that, in the second paragraph of the proof, it
is asserted that a result of Brauer implies that the number of $C_{G_{r}}(\beta)H_{\beta}/N$-stable
characters of $H_{\beta}/N$ is equal to the number of $C_{G_{r}}(\beta)H_{\beta}/N$-stable
conjugacy classes of $H_{\beta}/N$, where $N=\Ker\psi_{\beta}$.
However, the quoted result of Brauer holds only for characters/conjugacy
classes fixed by a single element in a group, and does not necessarily
apply to the whole group $C_{G_{r}}(\beta)H_{\beta}/N$. We remark
that by results of Glauberman and Isaacs (see \cite[(13.24)]{Isaacs}
the appropriate generalisation of Brauer's result holds for coprime
group actions, but may fail otherwise. Since $p$ divides the orders
of both $C_{G_{r}}(\beta)H_{\beta}/N$ and $H_{\beta}/N$, the crucial
step in Hill's proof which asserts the existence of an extension of
$\psi_{\beta}$ to $C_{G_{r}}(\beta)H_{\beta}$ remains unclear.

In addition to the split regular representations, there are many regular
representations which are not split, in particular the \emph{cuspidal}
representations, that is, those where $\bar{\beta}$ has irreducible
characteristic polynomial. In \cite{Hill_semisimple_cuspidal} Hill
gave a construction of so-called \emph{strongly semisimple} representations,
that is, those for which $\bar{\beta}$ is semisimple and $\beta_{l'}\in\mfg_{l'}$
has additive Jordan decomposition $\beta_{l'}=s+n$, with $n$ in
the centre of the algebra $C_{\mfg_{l'}}(s)$.
\begin{example}
Consider the function $\iota:\F_{q}\rightarrow\mfo_{r}$, which is
the multiplicative section extended by setting $\iota(0)=0$. This
induces an injective function $\mfg_{1}\rightarrow$$\mfg_{r}$. An
element in $\mfg_{r}$ is called \emph{semisimple} if it is the image
of a semisimple element in $\mfg_{1}$ under the map $\mfg_{1}\rightarrow$$\mfg_{r}$.
Then any $\beta\in\mfg_{r}$ has a unique Jordan decomposition $\beta=s+n$,
where $s$ is semisimple, $n$ is nilpotent and $sn=ns$ (see \cite[Proposition~2.3]{Hill_semisimple_cuspidal}).
If $n=0$, then $\beta$ is strongly semisimple, so in particular,
there are strongly semisimple representations which are not regular.
The strongly semisimple representations include the cuspidal ones
(see \cite[Proposition~4.4]{Hill_semisimple_cuspidal}). 
\end{example}
Hill's construction of strongly semisimple representations for $r$
odd is summarised in the following (cf.~\cite[Proposition~3.6]{Hill_semisimple_cuspidal})
result:
\begin{thm}
\label{thm:Hill-strong-ss}Let $\pi\in\Irr(G_{r}\mid\psi_{\beta})$
be strongly semisimple. Then there exists a $\rho\in\Irr(K^{l'}\mid\psi_{\beta})$
and an extension $\tilde{\rho}$ of $\rho$ to $G_{r}(\psi_{\beta})$
such that 
\[
\pi=\Ind_{G_{r}(\psi_{\beta})}^{G_{r}}\tilde{\rho}.
\]
\end{thm}
Note that the only non-trivial part of this theorem is that $\rho$
has an extension. In fact, it follows from the proof in \cite{Hill_semisimple_cuspidal}
that \emph{every} $\rho\in\Irr(K^{l'}\mid\psi_{\beta})$ extends to
$G_{r}(\psi_{\beta})$. Moreover, by Theorem~\ref{thm:Clifford}\,\ref{Clifford2},
distinct extensions of $\rho$ give rise to distinct representations~$\pi$.

The elements of $\Irr(K^{l'}\mid\psi_{\beta})$ are constructed in
\cite[Proposition~4.2\,(3)]{Hill_regular}, so that together with
the above theorem, this gives a complete construction of strongly
semisimple representations, up to a knowledge of the elements in $\Irr(G_{r}(\psi_{\beta})/K^{l'})\cong\Irr(C_{G_{l'}}(\beta_{l'}))$.
A version of Theorem~\ref{thm:Hill-strong-ss} holds also when $r$
is even; see \cite[Proposition~3.3]{Hill_semisimple_cuspidal}.

We see that out of the regular representations, Hill's constructions
cover at most those which are semisimple (i.e., where $\bar{\beta}$
is semisimple). The next step was taken recently by Takase, who proved
the following (see \cite[Theorem~3.2.2, 5.2.1 and 5.3.1]{Takase-16}):
\begin{thm}
\label{thm:Takase}Let $\pi\in\Irr(G_{r}\mid\psi_{\beta})$ be a regular
character and suppose that $\bar{\beta}$ satisfies at least one of
the following properties:

\begin{enumerate}
\item $\bar{\beta}$ has separable characteristic polynomial and $p>2$,
\item $\bar{\beta}$ has Jordan blocks of size at most $4$ and $p>7$.
\end{enumerate}
Then there exists a $\rho\in\Irr(K^{l'}\mid\psi_{\beta})$ and an
extension $\tilde{\rho}$ of $\rho$ to $G_{r}(\psi_{\beta})$ such
that $\pi=\Ind_{G_{r}(\psi_{\beta})}^{G_{r}}\tilde{\rho}$.
\end{thm}
Just as for Hill's theorem on strongly semisimple representations
above, the difficulty in Takase's proof lies in showing that every
$\rho\in\Irr(K^{l'}\mid\psi_{\beta})$ extends to $G_{r}(\psi_{\beta})$.
The existence of an extension follows from the vanishing of the cohomology
group $H^{2}(\F_{q}[\bar{\beta}]^{\times},\mathbb{C}^{\times})$,
the so-called Schur multiplier. When $\bar{\beta}$ has irreducible
characteristic polynomial, $\F_{q}[\bar{\beta}]$ is a finite field,
so $\F_{q}[\bar{\beta}]^{\times}$ is cyclic. In this case it is well
known that $H^{2}(\F_{q}[\bar{\beta}]^{\times},\mathbb{C}^{\times})$
is trivial. For $p>2$ Takase reduces the separable case to the irreducible,
and thus proves Theorem~\ref{thm:Takase} when $\bar{\beta}$ satisfies
the first condition; cf.~\cite[Theorem~4.3.2]{Takase-16}. The existence
of an extension when $\bar{\beta}$ satisfies the second condition
is proved in \cite{Takase-16} by explicit computation of the relevant
cocycles.

These results led Takase to conjecture that a certain element in the
Schur multiplier is always trivial for $p$ large enough; see \cite[Conjecture~4.6.5]{Takase-16}.

\section{\label{sec:KOS}The construction of Krakovski, Onn and Singla}

We will now describe the construction of regular representations of
$G_{r}$, $r$ odd, due to Krakovski, Onn and Singla \cite{KOS}.
This gives a construction of all the regular representations, provided
the residue characteristic $p$ of $\mfo$ is \emph{odd}. Furthermore,
\cite{KOS} also contains constructions and enumeration of all the
regular representations of $\SL_{N}(\mfo_{r})$ when $p>N$, and of
the unitary groups $\mathrm{SU}_{N}(\mfo_{r})$ and $\mathrm{GU}_{N}(\mfo_{r})$
with respect to a quadratic unramified extension of $\mfo$ (with
some restrictions on $p$). The construction in \cite{KOS} was inspired
by a construction of Jaikin-Zapirain for $\SL_{2}(\mfo_{r})$, $p>2$;
see \cite[Section~7]{Jaikin-zeta}. We continue to assume that $r=l+l'$
is odd. The following result is \cite[Theorem~3.1]{KOS}, which is
a more detailed statement of \cite[Theorem~A]{KOS}. We state this
only for $\GL_{N}$, in a form slightly adapted to our present notation.
\begin{thm}
\label{thm:KOS}Assume that $\mfo$ has residue characteristic $p>2$.
Let $\sigma\in\Irr(K^{l'}\mid\psi_{\beta})$ with $\beta$ regular.
Then $\sigma$ has an extension $\tilde{\sigma}$ to $G_{r}(\psi_{\beta})$,
and thus any $\pi\in\Irr(G_{r}\mid\psi_{\beta})$ is of the form $\pi=\Ind_{G_{r}(\psi_{\beta})}^{G_{r}}\tilde{\sigma}$,
for some extension $\tilde{\sigma}$. 
\end{thm}
In particular, this proves a strengthened form of Takase's conjecture
mentioned above, namely for all $p>2$ (another proof of this, for
all $p$, follows from the construction of Stasinski and Stevens).
We elaborate on the proof of \cite[Theorem~3.1]{KOS} in order to
provide some of the details of the construction. As we have already
remarked in previous sections, the main difficulty is to show that
every $\sigma\in\Irr(K^{l'}\mid\psi_{\beta})$ extends to $G_{r}(\psi_{\beta})$.
We will mainly formulate things in our present notation, but use the
notation of \cite{KOS} where possible.

\subsection{\label{subsec:KOS-Characters}Characters}

Assume that $p>2$. For $i$ such that $r/2\leq i<r$, the exponential
map $\exp:x\mapsto1+x$ gives an isomorphism $\mfp^{i}\mfg_{r}\rightarrow K^{i}$
(we already saw this in Section~\ref{sec:Clifford-theory}, and it
works for any $p$). Moreover, when $p>2$ and $r/3\leq i<r/2$, the
exponential map $\exp:x\mapsto1+x+\frac{1}{2}x^{2}$ gives a bijection
$\mfp^{i}\mfg_{r}\rightarrow K^{i}$, which is however not an isomorphism
in general. As usual, the inverse of this exponential map is given
by the logarithm $\log:1+x\mapsto x-\frac{1}{2}x^{2}$. Every $\beta\in\mfg_{r}$
defines a character
\[
\phi_{\beta}:\mfg_{r}\longrightarrow\C^{\times},\quad\text{where }\phi_{\beta}(x)=\psi(\varpi^{-r}\tr(\beta x)).
\]
The corresponding map $\beta\mapsto\phi_{\beta}$ is an isomorphism.
Any $\theta\in\Irr(\mfp^{l}\mfg_{r})$ can be pre-composed with the
logarithm map $\log:K^{l}\rightarrow\mfp^{l}\mfg_{r}$, $1+x\mapsto x$,
to give a character $\log^{*}\theta:=\theta\circ\log\in\Irr(K^{l}$),
such that, for $1+x\in K^{l}$, 
\[
(\log^{*}\theta)(1+x)=\phi_{\beta}(x),
\]
where $\beta$ is determined by $\theta$. Note that $\log^{*}\theta=\psi_{\beta}$,
where $\psi_{\beta}$ is as in (\ref{eq:character-psi-b}). In particular,
$\phi_{\beta}$ restricts to $\theta$ on $\mfp^{l}\mfg_{r}$, but
for a given $\theta$, there is more than one $\beta$ such that $\phi_{\beta}$
restricts to $\theta$, since the restriction only depends on $\beta$
mod~$\mfp^{l}$.

A crucial step in \cite{KOS} (due to Jaikin-Zapirain for $\SL_{2}$),
is to extend the above definition of $\log^{*}\theta$, in order to
give a useful description of certain characters on any subgroup $K^{l}\subseteq J_{\beta}\subseteq K^{l'}$
such that $J_{\beta}/K^{l}$ is a maximal isotropic subspace for the
form $B_{\beta}$ defined in Section~\ref{sec:Hill-Takase}. This
is the motivation behind \cite[Lemma~3.2]{KOS}, and the essential
reason for the assumption $p>2$. The following result gives a summary
of the key facts involved (see \cite[Lemma~3.2 and Section~3.2]{KOS}).
\begin{lem}
\label{lem:KOS-crucial}Let $J_{\beta}$ be such that $J_{\beta}/K^{l}$
is a maximal isotropic subspace. Let $\theta''$ be the restriction
of a character $\phi_{\beta}\in\Irr(\mfg_{r})$ to $\mfp^{l'}\mfg_{r}$.
Then the function $\log^{*}\theta'':K^{l'}\rightarrow\C^{\times}$
defines a multiplicative character when restricted to $J_{\beta}$. 
\end{lem}
\begin{proof}
Let $1+\varpi^{l'}x$ and $1+\varpi^{l'}y$ be elements in $J_{\beta}$.
Direct computation yields the commutator
\[
g:=[(1+\varpi^{l'}x),(1+\varpi^{l'}y)]=1+\varpi^{2l'}(xy-yx).
\]
Since we are assuming that $p>2$, we have a unique square root $g^{1/2}=1+\frac{1}{2}\varpi^{2l'}(xy-yx)$.
In particular, since $2l'=r-1$, $g^{1/2}$ is in the centre of $K^{l'}$.
Thus,
\begin{align*}
\log((1+\varpi^{l'}x)(1+\varpi^{l'}y)) & =\log((1+\varpi^{l'}x)(1+\varpi^{l'}y)g{}^{-1/2}g^{1/2})\\
 & =\log((1+\varpi^{l'}x)(1+\varpi^{l'}y)g{}^{-1/2})+\log(g^{1/2})\\
 & =\log(1+\varpi^{l'}(x+y)+\frac{1}{2}\varpi^{2l'}(xy+yx))+\log(g^{1/2})\\
 & =\log(1+\varpi^{l'}x)+\log(1+\varpi^{l'}y)+\log(g^{1/2}),
\end{align*}
where the second equality follows from the fact that $g^{1/2}$ is
central. Applying $\theta''$, we get
\begin{align*}
\theta''(\log((1+\varpi^{l'}x)(1+\varpi^{l'}y))) & =\theta''(\log(1+\varpi^{l'}x))+\theta''(\log(1+\varpi^{l'}y))\\
 & \phantom{{{}={}}}+\theta''(\log(g^{1/2}))\\
 & =\theta''(\log(1+\varpi^{l'}x))+\theta''(\log(1+\varpi^{l'}y))\\
 & \phantom{{{}={}}}+\psi(\frac{1}{2}\varpi^{-1}\tr(\beta(xy-yx)))\\
 & =\theta''(\log(1+\varpi^{l'}x))+\theta''(\log(1+\varpi^{l'}y)),
\end{align*}
where the last equality follows from the fact that $\tr(\bar{\beta}(\bar{x}\bar{y}-\bar{y}\bar{x}))=B_{\beta}(1+\varpi^{l'}x,1+\varpi^{l'}y)=0$,
since $J_{\beta}/K^{l}$ is isotropic.
\end{proof}
The crucial corollary of this lemma is that any $\log^{*}\theta\in\Irr(K^{l})$
extends to $J_{\beta}$ \emph{by the same formula}, that is, $\log^{*}\theta''=\log^{*}\phi_{\beta}$.
We emphasise that the key is not just that $\log^{*}\theta$ has an
extension to $J_{\beta}$ (this is true for any $p$, by \cite[Proposition~4.2]{Hill_regular}),
but that there is an extension given by an explicit formula which
makes it evident that the extension is stabilised by any $g\in C_{G_{r}}(\beta)$
which normalises $J_{\beta}$. As we will explain below, the $p$-Sylow
subgroup $P_{\beta}$ of $C_{G}(\beta)$ normalises $J_{\beta}$,
so the extension $\log^{*}\theta''$ of $\log^{*}\theta$ to $J_{\beta}$
is stabilised by $P_{\beta}$. Note that it is not known whether all
of $C_{G}(\beta)$ normalises $J_{\beta}$, in general.

We now describe the representations of the non-abelian group $K^{l'}$,
following Hill \cite[Proposition~4.2]{Hill_regular}. It is easy to
check that the radical of the bilinear form $B_{\beta}$ introduced
in Section~(\ref{sec:Hill-Takase}), is $(C_{G_{r}}(\beta)\cap K^{l'})K^{l}/K^{l}$.
There is then a subgroup $K^{l}\subseteq J_{\beta}\subseteq K^{l'}$
such that $J_{\beta}/K^{l}$ is a maximal isotropic subspace. The
radical and maximal isotropic subspace correspond to two subspaces
of $\M_{N}(\F_{q})\cong K^{l'}/K^{l}$, and we let
\[
\mfr_{\beta}\quad\text{and}\quad\mfj_{\beta}
\]
denote the inverse images in $\mfp^{l'}\mfg_{r}$ of these two subspaces,
respectively, under the map $\mfp^{l'}\mfg_{r}\rightarrow\M_{N}(\F_{q})$,
$\varpi^{l'}x\mapsto\bar{x}$. Clearly $\mfr_{\beta}$ and $\mfj_{\beta}$
only depend on $\bar{\beta}\in\mfg_{r}$. Let $\theta\in\Irr(\mfp^{l}\mfg_{r})$,
and let $\theta'$ be an extension of $\theta$ to $\mfr_{\beta}$
(here we are just talking about characters of abelian groups). Then
$\theta'$ determines a unique irreducible representation of $K^{l'}$,
which arises as follows. Let $\theta''$ be an extension of $\theta'$
to $\mfj_{\beta}$. Then $\log^{*}\theta''$ is a character of the
group $J_{\beta}$ thanks to Lemma~\ref{lem:KOS-crucial}, and 
\[
\Ind_{J_{\beta}}^{K^{l'}}(\log^{*}\theta'')
\]
can be shown to be irreducible. In fact, it is the unique element
in $\Irr(K^{l'}\mid\log^{*}\theta')$.

\subsection{Construction of representations}

From now on, let $\theta\in\Irr(\mfp^{l}\mfg_{r})$ be a character
that corresponds to a regular element, that is $\log^{*}\theta=\psi_{\beta}$
, where $\beta\in\mfg_{r}$ is regular (recall that $\psi_{\beta}$
only depends on the coset $\beta+\mfp^{l'}\mfg_{r}$). 
\begin{lem}
\label{lem:Stab-sigma-beta}Let $\sigma\in\Irr(K^{l'}\mid\log^{*}\theta)$.
Then $G_{r}(\sigma)=G_{r}(\psi_{\beta})$.
\end{lem}
\begin{proof}
Let $\theta'\in\Irr(\mfr_{\beta})$ be the unique extension of $\theta$
that corresponds to $\sigma$. Choose $\beta'\in\mfg_{r}$ such that
$\phi_{\beta'}\in\Irr(\mfg_{r})$ is an extension of $\theta'$. Then
$\phi_{\beta'}$ is also an extension of $\theta$, so $\beta'\equiv\beta\mod\mfp^{l'}\mfg_{r}$,
and by (\ref{eq:Stabiliser}) we have
\[
G_{r}(\sigma)\subseteq G_{r}(\log^{*}\theta)=C_{G_{r}}(\beta)K^{l'}=C_{G_{r}}(\beta')K^{l'},
\]
where the first inclusion follows from the fact that $\log^{*}\theta$
is the unique irreducible character of $K^{l}$ contained in $\sigma$
(the orbit of the restriction of $\sigma$ to $K^{l}$ consists of
copies of $\psi_{\beta}$ since $K^{l'}$ stabilises $\psi_{\beta}$).

For the reverse inclusion, note that $C_{G_{r}}(\beta')$ stabilises
$\phi_{\beta'}$, hence its restriction $\theta'$, and hence the
character $\log^{*}\theta'$. Since $\sigma$ is the unique representation
in $\Irr(K^{l'}\mid\log^{*}\theta')$, $\sigma$ is stabilised by
$C_{G_{r}}(\beta')$, and so $C_{G_{r}}(\beta')K^{l'}\subseteq G_{r}(\sigma)$.
\end{proof}
We now explain how to show that $\sigma$ extends to the stabiliser
$G_{r}(\psi_{\beta})$. For this, it will be enough (by Lemma~\ref{lem:Stab-sigma-beta}
and \cite[Corollary~11.31]{Isaacs}) to show that $\sigma$ extends
to the $p$-Sylow subgroup of $G_{r}(\psi_{\beta})$ (which is unique
since $G_{r}(\psi_{\beta})$ is abelian modulo the $p$-group $K^{l'}$).
Let $P_{\beta}$ denote the $p$-Sylow subgroup of $C_{G_{r}}(\beta)$.
The following crucial lemma, see \cite[Lemma~3.4]{KOS}, goes back
to Howe:
\begin{lem}
Let $V$ be a finite dimensional $\F_{p}$-vector space and $\alpha$
an antisymmetric bilinear form on $V$. Suppose that $P$ is a $p$-group
which acts on $V$ and preserves $\alpha$. Then there exists a maximal
isotropic subspace $U$ of $V$ which is $P$-invariant.
\end{lem}
The group $P_{\beta}$ acts on $K^{l'}$ and $K^{l}$ by conjugation,
and hence induces an action on the vector space $K^{l'}/K^{l}$. By
the above lemma, there exists a maximal isotropic subspace of $K^{l'}/K^{l}$
which is stable under this action of $P_{\beta}$, that is, there
is a subgroup $K^{l}\subseteq J_{\beta}\subseteq K^{l'}$, such that
the image of $J_{\beta}$ in $K^{l'}/K^{l}$ is a maximal isotropic
subspace and such that $J_{\beta}$ is normalised by $P_{\beta}$.
As in the proof of Lemma~\ref{lem:Stab-sigma-beta}, let $\theta'\in\Irr(\mfr_{\beta})$
be the unique extension of $\theta$ that corresponds to $\sigma$
and $\phi_{\beta}\in\Irr(\mfg_{r})$ an extension of $\theta'$. Then
the restriction $\phi_{\beta}|_{\mfj_{\beta}}$ is stabilised by $P_{\beta}$
(because it is stabilised by all of $C_{G_{r}}(\beta)$), and thus
\[
\log^{*}(\phi_{\beta}|_{\mfj_{\beta}})
\]
is a character of $J_{\beta}$ (by Lemma~\ref{lem:KOS-crucial}),
which is stabilised by $P_{\beta}$. Here we again see the crucial
role played by Lemma~\ref{lem:KOS-crucial} as well as the order
in which choices are made: For any $\sigma\in\Irr(K^{l'}\mid\psi_{\beta})$,
there is a unique $\theta'\in\Irr(\mfr_{\beta})$, and this extends
to a $\theta''\in\Irr(\mfj_{\beta})$ such that $\log^{*}\theta''\in\Irr(J_{\beta})$
is stabilised by $C_{G_{r}}(\beta)$.

Since $\log^{*}(\phi_{\beta}|_{\mfj_{\beta}})$ is one-dimensional
and $P_{\beta}$ is abelian, this character extends to a character
$\omega\in\Irr(P_{\beta}J_{\beta})$. The induced representation 
\[
\sigma':=\Ind_{P_{\beta}J_{\beta}}^{P_{\beta}K^{l'}}\omega
\]
has dimension
\[
[P_{\beta}K^{l'}:P_{\beta}J_{\beta}]=\frac{|P_{\beta}|\cdot|K^{l'}|/|P_{\beta}\cap K^{l'}|}{|P_{\beta}|\cdot|J_{\beta}|/|P_{\beta}\cap J_{\beta}|}=\frac{|P_{\beta}\cap J_{\beta}|}{|P_{\beta}\cap K^{l'}|}[K^{l'}:J_{\beta}].
\]
Since $J_{\beta}$ contains the group $(C_{G_{r}}(\beta)\cap K^{l'})K^{l}$
(since every maximal isotropic subspace contains the radical of the
form), we have $P_{\beta}\cap J_{\beta}\supseteq P_{\beta}\cap K^{l'}$.
The reverse inclusion is trivial, so we have $\dim\sigma'=[K^{l'}:J_{\beta}]=\dim\sigma$.
Since $\sigma'$ must contain $\sigma$ on restriction to $K^{l'}$
(because $\sigma'$ contains $\log^{*}\theta'$), $\sigma'$ must
be an extension of $\sigma$ (so in particular, $\sigma'$ must be
irreducible). Thus $\sigma$ extends to the $p$-Sylow in $G_{r}(\psi_{\beta})$
and hence to all of $G_{r}(\psi_{\beta})$, by the above remarks.
This concludes the proof of Theorem~\ref{thm:KOS}.

\section{\label{sec:SS}The construction of Stasinski and Stevens}

In this section we summarise forthcoming work of Stasinski and Stevens
\cite{SS/2016} which gives a construction of all the regular representations
of $G_{r}=\GL_{N}(\mfo_{r})$, without any restriction on the residue
characteristic. As in the previous two sections, we assume that $r=l+l'$
is odd. 

One of the key distinguishing features of the present approach is
the systematic use of the subgroup structure of $G_{r}$ provided
by lattice chains. In particular, for a given regular orbit, two specific
associated parahoric subgroups and their filtrations will play a crucial
role. The construction is somewhat analogous to the construction of
supercuspidal representations of Bushnell and Kutzko \cite{BushnellKutzko},
but with the difference that for us everything takes place inside
$G_{r}$ and all relevant centralisers are abelian (because we consider
only regular representations).

\subsection{Subgroup structure\label{subsec:Subgroup-structure}}

Let $\mfA\subseteq\mfg_{r}=\M_{N}(\mfo_{r})$ be a parahoric subalgebra,
that is, the preimage under the reduction mod $\mfp$ map of a parabolic
subalgebra of $\mfg_{1}=\M_{N}(\F_{q})$. Let $\mfP$ denote the preimage
of the corresponding nilpotent radical of the parabolic subalgebra.
A parabolic subalgebra of $\mfg_{1}$ is the stabiliser of a flag,
and as such is $G_{1}$-conjugate to a block upper triangular subalgebra
of $\mfg_{1}$. The nilpotent radical of a parabolic subalgebra in
block form is the subalgebra obtained by replacing each diagonal block
by a $0$-block of the same size. Define the following subgroups of
$G_{r}$: 
\[
U=U^{0}=\mfA^{\times},\quad U^{m}=1+\mathfrak{\mfP}^{m},\text{ for }m\geq1.
\]
Let $e=e(\mfA)$ be the length of the flag in $\mfg_{1}$ defining
$\mfA$. Then it can be shown that
\begin{equation}
\mfp\mfA=\mfA\mfp=\mfP^{e}\label{eq:pA=00003DPe}
\end{equation}
and one can think of $e$ as a ramification index. We have a filtration
\[
U\supset U^{1}\supset\dots\supset U^{er-1}\supset U^{er}=\{1\},
\]
where the inclusions can be shown to be strict. It is also convenient
to define $U^{i}=\{1\}$ for all $i>er$. Since $\mathfrak{P}$ is
a (two-sided) ideal in $\mathfrak{A}$, each group $U^{i}$ is normal
in $U$. Moreover, we have the commutator relation 
\[
[U^{i},U^{j}]\subseteq U^{i+j}.
\]
Thus in particular, the group $U^{i}$ is abelian whenever $i\geq er/2$.

From now on, let $\beta\in\mfg_{r}$ be a regular element and write
$\bar{\beta}$ for its image in $\mfg_{1}$. We will associate a certain
parahoric subalgebra to $\beta$ (or rather, to the orbit of $\bar{\beta}$),
which will be denoted by $\Amin$. Let 
\[
\prod_{i=1}^{h}f_{i}(x)^{m_{i}}\in\F_{q}[x]
\]
be the characteristic polynomial of $\bar{\beta}$, where the $f_{i}(x)$
are distinct and irreducible of degree $d_{i}$, for $i=1,\dots,h$.
This determines a partition of $n$:
\[
\lambda=(d_{1}^{m_{1}},\dots,d_{h}^{m_{h}})=(\underbrace{d_{1},d_{1},\dots,d_{1}}_{m_{1}\text{ times}},\dots,\underbrace{d_{h},d_{h},\dots,d_{h}}_{m_{h}\text{ times}}).
\]
We define $\Amin\subseteq\mfg_{r}$ to be the preimage of the standard
parabolic subalgebra of $\mfg_{1}$ corresponding to $\lambda$ (i.e.,
the block upper-triangular subalgebra whose block sizes are given
by $\lambda$, in the order given above). Moreover, we let $\Amax=\mfg_{r}=\M_{N}(\mfo_{r})$
be the full matrix algebra. Let $\Pmin$ and $\Pmax$ be the corresponding
ideals in $\Amin$ and $\Amax$, respectively. For $*\in\{\mathrm{m},\mathrm{M}\}$
we have the corresponding groups 
\[
U_{*}=U_{*}^{0}=\mfA_{*}^{\times},\quad U_{*}^{i}=1+\mfP_{*}^{i},\qquad\text{for }i\geq1,
\]
and the filtration 
\[
U_{*}\supset U_{*}^{1}\supset\dots\supset U_{*}^{e_{*}r}=\{1\},
\]
where $e_{*}=e(\mfA_{*})$. Note that $\Umax^{i}=K^{i}$. and $\emax=1$.
The label $\mathrm{m}$ here stands for ``minimal'', while $\mathrm{M}$
stands for ``maximal''. From the definitions, we have 
\begin{align*}
\Umin/\Umin^{1} & \cong\prod_{i=1}^{h}\GL_{d_{i}}(\F_{q}))^{m_{i}},\\
\Umax/\Umax^{1} & \cong\GL_{N}(\F_{q}).
\end{align*}
Note that if $\bar{\beta}$ has irreducible characteristic polynomial,
then $\Amin=\mfg_{r}$, and $\Umin^{i}=K^{i}$ are the normal subgroups
defined earlier. 

By definition, we have $\Amax\supseteq\Amin$, and therefore $\Pmin\supseteq\Pmax$.
The relations $\Amax\supseteq\Amin\supseteq\Pmin\supseteq\Pmax$ imply
that for every $i\geq1$, $\Pmax^{i}$ is a two-sided ideal in $\Amin$,
so $\Umin$ normalises $\Umax^{i}$. For $*\in\{\mathrm{m},\mathrm{M}\}$,
we can therefore define the following groups
\begin{align*}
C & =C_{G_{r}}(\beta),\\
J_{*} & =(C\cap U_{*})U_{*}^{e_{*}l'},\\
J_{*}^{1} & =(C\cap U_{*}^{1})U_{*}^{e_{*}l'},\\
H_{*}^{1} & =(C\cap U_{*}^{1})U_{*}^{e_{*}l'+1}.
\end{align*}
Recall that since $\beta$ is regular, $C$ is abelian. Since $[U_{*}^{1},U_{*}^{e_{*}l'}]\subseteq U_{*}^{e_{*}l'+1}$
and $\mfA_{*}^{\times}$ normalises $U_{*}^{e_{*}l'}$, the group
$J_{*}$ normalises both $J_{*}^{1}$ and $H_{*}^{1}$. Moreover,
we define the group
\[
\JmM=(C\cap\Umin^{1})K^{l'}.
\]
We have the following diagram of subgroups, where the vertical and
slanted lines denote inclusions (we have only indicated the inclusions
which are relevant to us and repeat the definitions of the groups,
for the reader's convenience).$$
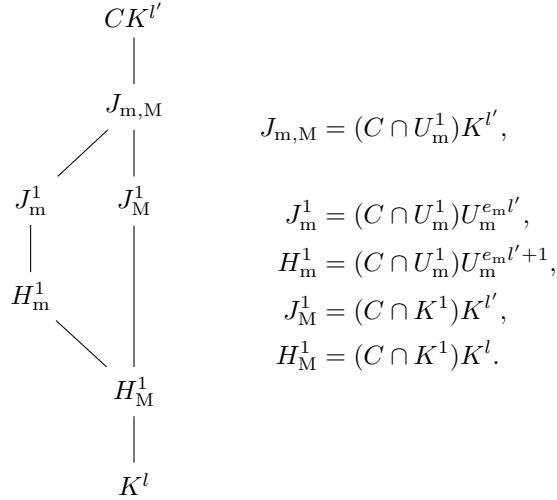
\begin{tikzcd}[column sep=0.4cm] 
{} & CK^{l'}\arrow[dash]{d}\\
{} &  \JmM\arrow[dash]{dl}\arrow[dash]{d}\\
\Jmin^1\arrow[dash]{d} & \Jmax^1\arrow[dash]{dd}\\ 
\Hmin^1\arrow[dash]{dr} & {}  \\ 
{}  & \Hmax^1\arrow[dash]{d}\\
{}  & K^l
\end{tikzcd}
\qquad\quad
\begin{aligned} 
\JmM &= (C\cap\Umin^{1})K^{l'},\\
\\
\Jmin^{1} & =(C\cap \Umin^{1})\Umin^{\emin l'},\\
\Hmin^{1} & =(C\cap \Umin^{1})\Umin^{\emin l'+1},\\ 
\Jmax^{1} & =(C\cap K^{1})K^{l'},\\
\Hmax^{1} & =(C\cap K^{1})K^{l}.\end{aligned}
$$We explain the non-trivial inclusions in the above diagram. Since
$\Pmin\supseteq\Pmax$, we have $\Umin^{1}\supseteq K^{1}$ and 
\[
\Umin^{e_{\mathrm{m}}l'+1}=1+\mfp^{l'}\Pmin\supseteq1+\mfp^{l'}\Pmax=K^{l};
\]
thus $\Hmin^{1}\supseteq\Hmax^{1}$. Moreover, 
\[
\Umin^{e_{\mathrm{m}}l'}=1+\mfp^{l'}\Amin\subseteq1+\mfp^{l'}\Amax=K^{l'},
\]
so $\JmM$ contains both $\Jmin^{1}$ and $\Jmax^{1}$ as subgroups.
We remark that $\Jmax^{1}$ is normal in $CK^{l'}$ since $C$ normalises
both $K^{1}$ and $K^{l'}$, and $[K^{l'},K^{1}]\subseteq K^{l}\subseteq K^{l'}$. 

The following lemma will be used in Step~\ref{Step: extn-JmM-JM}
of the construction we will outline below, and is the main reason
why we work with the algebra $\Amin$ and its associated subgroups. 
\begin{lem}
\label{lem:normal-p-Sylow}There exists a $G_{r}$-conjugate of $\beta$
such that the group $\JmM$ is a normal $p$-Sylow subgroup of $CK^{l'}$.
\end{lem}
We sketch the proof of this lemma. We first show that $\JmM$ is normal
in $\Jmax$. Since $C\cap\Amax^{\times}$ normalises $\JmM$ ($C$
being abelian), it is enough to observe that  $\Umax^{e_{\mathrm{M}}l'}$
normalises $\JmM$ (in any finite group $G$ with a normal subgroup
$N$ and a subgroup $H$, the group $HN$ is normalised by $N$; here
$G$ would be $\Umax$). Write $\beta_{\mathrm{m}}$ for the image
of $\beta$ in $\Umin/\Umin^{1}$. Then, up to conjugating $\beta$,
we have
\[
\beta_{\mathrm{m}}=\underbrace{\beta_{1}\oplus\dots\oplus\beta_{1}}_{m_{1}\text{ times}}\oplus\dots\oplus\underbrace{\beta_{h}\oplus\dots\oplus\beta_{h}}_{m_{h}\text{ times}},
\]
where $\beta_{i}\in\M_{d_{i}}(\F_{q})$, and $d_{i}$ and $m_{i}$
are as in the partition $\lambda$ above. With $\beta_{\mathrm{m}}$
of the above form, one can show that $\beta$ being regular implies
that $C\subseteq\Umin$, so we have an isomorphism
\[
CK^{l'}/\JmM\cong\frac{C}{(C\cap\Umin^{1})(C\cap K^{l'})}=\frac{C\cap\Umin}{(C\cap\Umin^{1})}.
\]
Then the isomorphism $\Umin/\Umin^{1}\cong\prod_{i=1}^{h}\GL_{d_{i}}(\F_{q})^{m_{i}}$
induces an isomorphism
\[
\frac{C\cap\Umin}{C\cap\Umin^{1}}\cong\prod_{i=1}^{h}C_{\GL_{d_{i}}(\F_{q})}(\beta_{i})^{m_{i}}.
\]
Each $\beta_{i}$ has irreducible characteristic polynomial over $\F_{q}$,
so $\F_{q}[\beta_{i}]/\F_{q}$ is an extension of degree $d_{i}$.
Since $C_{\GL_{d_{i}}(\F_{q})}(\beta_{i})=\F_{q}[\beta_{i}]^{\times}$,
we conclude that $p$ does not divide the order of $C_{\GL_{d_{i}}(\F_{q})}(\beta_{i})$.
Therefore, $p$ does not divide the order of $\frac{C}{C\cap\Umin^{1}}$,
so $\JmM$ is a $p$-Sylow subgroup of $CK^{l'}$ (in fact the unique
$p$-Sylow subgroup, since it is normal).

\subsection{Characters}

Let $\psi:F\rightarrow\mathbb{C}^{\times}$ be as in Section~\ref{sec:Clifford-theory}.
Let $\mathfrak{A},\mathfrak{P}$, and $U^{m},m\geq0$ be the objects
associated to an arbitrary flag of length $e$, as in Section~\ref{subsec:Subgroup-structure}.
Let $n$ and $m$ be two integers such that $e(r-1)+1\geq n>m\geq n/2>0$.
Then $U^{m}/U^{n}$ is abelian, and we have an isomorphism 
\[
\mathfrak{P}^{m}/\mathfrak{P}^{n}\longiso U^{m}/U^{n},\qquad x+\mathfrak{P}^{n}\longmapsto(1+x)U^{n}.
\]
Each $a\in\mfg_{r}$ defines a character $\mfg_{r}\rightarrow\mathbb{C}^{\times}$
via $x\mapsto\psi(\Tr(ax))$, and this defines an isomorphism $\mfg_{r}\rightarrow\Irr(\mfg_{r})$.
For any subgroup $S$ of $\mfg_{r}$, define 
\[
S^{\perp}=\{x\in\mfg_{r}\mid\psi(\Tr(xS))=1\}.
\]
Using the isomorphism $\mfg_{r}\rightarrow\Irr(\mfg_{r})$, we can
identify $S^{\perp}$ with the group of characters of $\mfg_{r}$
which are trivial on $S$. 

For any $\beta\in\mathfrak{P}^{e(r-1)+1-n}$ define a character $\psi_{\beta}:U^{m}\rightarrow\mathbb{C}^{\times}$
by 
\[
\psi_{\beta}(1+x)=\psi(\varpi^{-r}\Tr(\beta x)).
\]

\begin{lem}
\label{lem:characters}Let $e(r-1)+1\geq n>m\geq n/2>0$. Then

\begin{enumerate}
\item \label{enu:ortho}For any integer $i$ such that $0\leq i\leq e(r-1)+1$,
we have 
\[
(\mathfrak{P}^{i})^{\perp}=\mathfrak{P}^{e(r-1)+1-i}.
\]
\item \label{enu:chariso}The map $\beta\mapsto\psi_{\beta}$ induces an
isomorphism 
\[
\mathfrak{P}^{e(r-1)+1-n}/\mathfrak{P}^{e(r-1)+1-m}\longiso\Irr(U^{m}/U^{n}).
\]
\end{enumerate}
\end{lem}
We omit the proof of this lemma, and only remark that the first part
essentially follows from the observation that $j=e(r-1)+1$ is the
smallest integer such that $\mfP^{j}$ is strictly block-upper triangular
mod~$\mfp^{r}$. Indeed, $\mfP^{e(r-1)+1}=\mfp^{r-1}\mfP,$ and $\mfP$
is strictly block-upper mod $\mfp$. This implies that $\mfP^{\perp}=\mfP^{e(r-1)+1}$,
and the general case follows similarly. 

As a special case of the above, suppose that $e=1$, so that $\mfA=\mfg_{r}$
and $U^{m}=K^{m}=1+\mfp^{m}\mfg_{r}$. For any $r=n>m\geq r/2$ and
$\beta\in\mfg_{r}$, we have a character $\psi_{\beta}:K^{m}\rightarrow\mathbb{C}^{\times}$
defined as above, and the isomorphism of Lemma~\ref{lem:characters}\,\ref{enu:chariso}
becomes 
\begin{align*}
\mfg_{r}/\mathfrak{p}^{r-m}\mfg_{r} & \longiso\Irr(K^{m}),
\end{align*}
which agrees with the considerations in Section~\ref{sec:Clifford-theory}.

\subsection{Construction of representations}

For our fixed arbitrary regular element $\beta\in\mfg_{r}$, we start
with the character $\psi_{\beta}$ of $K^{l}$, and construct all
the irreducible representations of $CK^{l}$ which contain $\psi_{\beta}$.
Theorem~\ref{thm:Clifford}\,\ref{Clifford2} then yields all the
irreducible representations of $G_{r}$ with $\beta$ in their orbits.
The construction consists of a number of steps. For each step we indicate
some of the details involved.

Some of the steps can be carried out for the groups arising from the
algebras $\Amin$ and $\Amax$ simultaneously. For this purpose, we
will let $\mfA$ denote either $\Amin$ or $\Amax$, and let $\mfP$
be the radical in $\mfA$, with ``ramification index'' $e$. The
associated subgroups will be denoted by $U^{i}$, $H^{1}$, $J^{1}$.

\subsubsection*{Step 1:}

Show that $\psi_{\beta}$ has an extension $\thetaM$ to $\Hmax^{1}$.
Show that $\thetaM$ has an extension $\thetam$ to $\Hmin^{1}$.\\
\\
By Lemma~\ref{lem:characters}\,\ref{enu:chariso}, if we take
\[
m=el'+1,\qquad n=2m-1=e(r-1)+1,
\]
then $\beta$, or rather the coset $\beta+\mathfrak{\mfP}_{}^{e_{}l'}$,
defines a character on $U^{m}$, trivial on $U_{}^{n}$ by the same
formula as the one defining $\psi_{\beta}$. Since $\mathfrak{\mfP}_{}^{e_{}l'}=\mfp^{l'}\mfA_{}$,
we have a map
\[
\mfA_{}/\mathfrak{\mfP}_{}^{e_{}l'}\longrightarrow\mfg_{r}/\mfp^{l'}\mfg_{r},
\]
which sends the coset $\beta+\mathfrak{\mfP}_{}^{e_{}l'}$ to $\beta+\mfp^{l'}\mfg_{r}$.
Thus the different choices of lift of the latter coset give the different
choices of extension of $\psi_{\beta}$ to $U_{}^{e_{}l'+1}$. Our
element $\beta\in\mfA$ therefore gives rise to an extension (which
we still denote by $\psi_{\beta})$ of $\psi_{\beta}$ to $U^{el'+1}$,
defined by 
\[
\psi_{\beta}(1+x)=\psi(\varpi^{-r}\tr(\beta x)),\quad\text{for }x\in\mfP_{}^{e_{\mathrm{}}l'+1}.
\]
We now show the existence of the extensions $\thetaM$ and $\thetam$.
If $c\in C\cap U_{}^{1}$ and $x\in\mathfrak{\mfP}_{}^{e_{}l'+1}$,
then 
\begin{align*}
[c,1+x] & \in c(1+x)c^{-1}(1-x+\mfP_{}^{e_{}(r-1)+2})\\
 & =1+cxc^{-1}-x+\mfP_{}^{e_{}(r-1)+2}.
\end{align*}
By Lemma~\ref{lem:characters}\,\ref{enu:ortho}, since $\beta\in\mfA_{}$,
we have 
\begin{equation}
U_{}^{e_{}(r-1)+1}\subseteq\Ker\psi_{\beta},\label{eq:Umin-in-ker}
\end{equation}
so 
\[
\psi_{\beta}([c,1+x])=\psi(\varpi^{-r}\tr(\beta(cxc^{-1}-x)))=\psi(\varpi^{-r}\tr(c\beta xc^{-1}-\beta x))=1,
\]
where we have used that $c$ commutes with $\beta$.

Thus $C\cap U_{}^{1}$ stabilises the character $\psi_{\beta}$ on
$U^{el'+1}$, and since $C\cap U^{1}$ is abelian, this implies that
$\psi_{\beta}$ extends to $H^{1}=(C\cap U^{1})U^{e_{}l'+1}$. We
fix an extension $\thetaM$ to $\Hmax^{1}$ and an extension of $\thetaM$
to $\Hmin^{1}$, denoted $\thetam$.

\subsubsection*{Step 2:}

For $*\in\{\mathrm{m},\mathrm{M}\}$, construct the irreducible representations
$\eta_{*}$ of $J_{*}^{1}$ containing $\theta_{*}$. In particular,
show that there exists a unique representation $\etaM$ of $\Jmax^{1}$
containing $\thetaM$. \\
\\
As in the previous step, we will treat both cases simultaneously,
denoting either $\thetam$ or $\thetaM$ by $\theta$. We outline
the ingredients needed for this. First note that $\theta_{}$ is stabilised
by $J_{}^{1}$: Indeed, it is enough to show that $U_{}^{e_{}l'}$
stabilises $\theta_{}$. For $x\in\mfP_{}^{e_{}l'}$, $c\in(C\cap U_{}^{1})$
and $y\in\mfP_{}^{e_{}l'+1}$, we have
\begin{align*}
[1+x,c(1+y)] & \in(1+x)c(1+y)(1-x+x^{2}+\mfP^{e(r-1)+1})(1-y+\mfP^{e_{}(r-1)+1})c^{-1}\\
 & \subseteq(c+xc-cx+cy)(1-y)c^{-1}+\mfP^{e(r-1)+1}\\
 & \subseteq1+x-cxc^{-1}+\mfP^{e_{}(r-1)+1}.
\end{align*}
Hence, since $\psi_{\beta}$ is trivial on $U^{e(r-1)+1}$ (see (\ref{eq:Umin-in-ker}))
and $c$ commutes with $\beta$, we have 
\[
\theta([1+x,c(1+y)])=\psi_{\beta}([1+x,c(1+y)])=\psi(\varpi^{-r}\tr(c\beta xc^{-1}-\beta x))=1.
\]
Next, we have 
\[
J^{1}/H^{1}\cong\frac{U^{e_{}l'}}{(C\cap U^{e_{}l'})U^{e_{}l'+1}},
\]
and $U^{e_{}l'}/U^{e_{}l'+1}$ is isomorphic to a subgroup of $\mfg_{1}=\M_{N}(\F_{q})$.
Thus $J^{1}/H^{1}$ is a quotient of an elementary abelian $p$-group
and has the structure of a finite dimensional $\F_{q}$-vector space.
Define the alternating bilinear form 
\[
h_{\beta}:J^{1}/H^{1}\times J^{1}/H^{1}\longrightarrow\mathbb{C}^{\times},\qquad h_{\beta}(xH^{1},yH^{1})=\theta([x,y])=\psi_{\beta}([x,y]).
\]
Note that $[J^{1},J^{1}]\subseteq U^{el'+1}$, so we have $\theta([x,y])=\psi_{\beta}([x,y])$. 

Let $\overline{R}_{\beta}$ be the radical of the form $h_{\beta}$,
and let $\overline{W}_{\beta}$ be a maximal isotropic subspace (if
we need to specify which parabolic subalgebra $\mfA_{*}$ we are working
with, we will write $\overline{R}_{\beta,*}$ and $\overline{W}_{\beta,*}$,
for $*\in\{\mathrm{m},\mathrm{M}\}$). Let $R_{\beta}$ and $W_{\beta}$
denote the preimages of $\overline{R}_{\beta}$ and $\overline{W}_{\beta}$
under the map $J^{1}\rightarrow J^{1}/H^{1}$, respectively. For our
purposes, we need to determine the order of the group $W_{\beta}$
and this can be done by determining the order of $R_{\beta}$, or
equivalently, the dimension of $\overline{R}_{\beta}$ (as a vector
space over $\F_{q}$). Consider the map
\[
\rho:U^{el'}\longrightarrow U^{el'}/U^{el'+1}\longiso\mfA/\mfP,
\]
where the isomorphism is given by $(1+\varpi^{l'}x)U^{e_{\mathrm{}}l'+1}\mapsto x+\mfP$.
Let $\bar{\beta}_{\mathrm{}}$ denote the image of $\beta$ in $\mfA/\mfP$
under this map. One can then show that 
\[
R_{\beta}=(C\cap U^{1})\cdot\rho^{-1}(C_{\mfA/\mfP}(\bar{\beta})).
\]
A general result then says that there exists an extension $\theta'$
of $\theta$ to $R_{\beta}$, and, for each such extension $\theta'$,
a unique $\eta\in\Irr(J^{1}\mid\theta')$. Indeed, one shows that
there exists an extension $\theta''$ of $\theta'$ to $W_{\beta}$,
that $\eta:=\Ind_{W_{\beta}}^{J^{1}}\theta''$ is irreducible and
that $\eta$ is independent of the choice of extension $\theta''$
to $W_{\beta}$ (cf.~Section~\ref{subsec:KOS-Characters}). In particular,
it turns out that $R_{\beta,\mathrm{M}}\subseteq\Hmax^{1}$, so there
is no choice for $\theta'$ in this case, and hence there exists a
unique $\etaM\in\Irr(\Jmax^{1}\mid\thetaM)$. 

\subsubsection*{Step 3:}

Show that there exists an extension $\hatetaM$ of $\etaM$ to $\Jmax$.
\\
\\
This step can be seen as the reason for involving the ``auxiliary''
path through $\Hmin^{1}$ and $\Jmin^{1}$. In the previous step,
we constructed an irreducible representation $\eta_{\mathrm{m}}$
of $\Jmin^{1}$ containing $\theta_{\mathrm{m}}$. We now need to
determine the dimension of the induced representation 
\[
\eta:=\Ind_{\Jmin^{1}}^{\JmM}\eta_{\mathrm{m}}.
\]
The order of $R_{\beta,\mathrm{m}}$, can be used to calculate the
dimension of $\etam$, indeed $\dim\etam=[\Jmin^{1}:R_{\beta,\mathrm{m}}]^{1/2}$,
so
\[
\dim\eta=[\Jmin^{1}:R_{\beta,\mathrm{m}}]^{1/2}[\JmM:\Jmin^{1}].
\]
 Comparing this with the dimension of $\etaM$, which is $[\Jmax^{1}:\Hmax^{1}]^{1/2}=q^{N(N-1)/2}$,
it turns out that $\dim\etaM=\dim\eta$. Then, since the restricted
representation $\eta|_{\Jmax^{1}}$ contains $\thetaM$ on further
restriction to $\Hmax^{1}$, and $\etaM$ is the unique representation
of $\Jmax^{1}$ with this property, it follows that $\eta$ contains
$\etaM$ on restriction to $\Jmax^{1}$. The equality of the dimensions
then forces $\eta|_{\Jmax^{1}}=\etaM$ (and in particular, $\eta$
is irreducible). 

Furthermore, one can show that all of $CK^{l'}$ stabilises the character
$\thetaM$. Since $\JmM$ is a $p$-Sylow subgroup in $CK^{l'}$ by
Lemma~\ref{lem:normal-p-Sylow} and $\etaM$ extends to $\JmM$,
it then follows from \cite[Corollary~11.31]{Isaacs} and a theorem
of Gallagher \cite[Theorem~6]{Gallagher} that $\etaM$ has an extension
$\hatetaM$ to $CK^{l'}$ (the same extension result was used in the
end of Section~\ref{sec:KOS} for the extension from $J_{\beta}$
to $C_{G_{r}}$).

Note that $\eta_{\mathrm{m}}$ is not the only representation containing
$\theta_{\mathrm{m}}$, and therefore $\eta$ is not unique. This
does not matter for us, since we are only interested in proving that
$\etaM$ has an extension, so we only need one representation $\eta$.

We also remark that even though both $\eta$ and $\hatetaM$ are extensions
of $\etaM$, we do not know (and do not need to know) whether $\hatetaM$
is an extension of $\eta$. 

\subsubsection*{Step 4:}

The final step in the construction is to note that every irreducible
representation of $CK^{l'}$ which contains $\psi_{\beta}$ is of
the form $\hatetaM$ for some choice of extension $\thetaM$ of $\psi_{\beta}$
and some choice of extension $\hatetaM$ of $\etaM$, and that distinct
choices of $\thetaM$, as well as distinct choices of extensions $\hatetaM$
of $\etaM$, give rise to distinct representations of $CK^{l'}$. 

By a standard result in Clifford theory (Lemma~\ref{thm:Clifford})
we have a one to one correspondence between $\Irr(CK^{l'}\mid\psi_{\beta})$
and $\Irr(G_{r}\mid\psi_{\beta})$ given by induction. Thus, we have
constructed all the irreducible representations of $G_{r}$ with $\beta$
in their orbits.

Schematically, the construction is illustrated by the following diagrams
(dotted lines are extensions, dashed are Heisenberg lifts, and solid
one between $\etam$ and $\eta$ is an induction):

$$
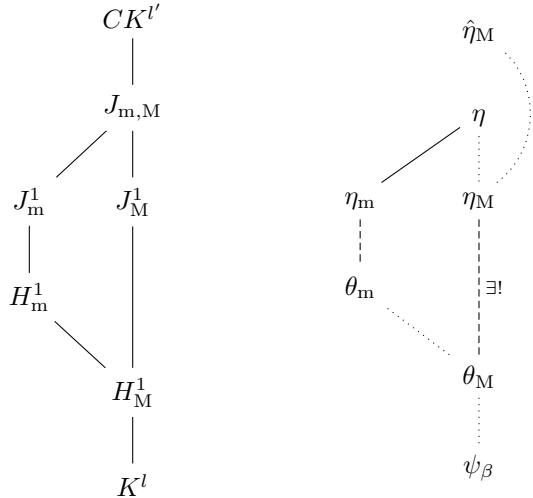
\begin{tikzcd}[column sep=0.4cm] 
{} & CK^{l'}\arrow[dash]{d}\\
{} &  \JmM\arrow[dash]{dl}\arrow[dash]{d}\\
\Jmin^1\arrow[dash]{d} & \Jmax^1\arrow[dash]{dd}\\ 
\Hmin^1\arrow[dash]{dr} & {}  \\ 
{}  & \Hmax^1\arrow[dash]{d}\\
{}  & K^l
\end{tikzcd}
\qquad\qquad\qquad
\begin{tikzcd}
{} &  \hatetaM\\
{} &  \eta\arrow[dash]{dl}\arrow[dotted, no head]{d}\\
\etam\arrow[dashed, no head]{d} & \etaM\arrow[dashed, no head, "\exists !"]{dd}\arrow[uu, dotted, no head, bend right=50]\\ 
\theta_{\mathrm{m}}\arrow[dotted, no head]{dr} & {}  \\ 
{}  & \thetaM\arrow[dotted, no head]{d}\\
{}  & \psi_{\beta}
\end{tikzcd}
$$

\section{\label{sec:Open-problems}Open problems}

We close with a non-exhaustive list of open problems in the representation
theory of $G_{r}=\GL_{N}(\mfo_{r})$. Several other problems are suggested
in \cite[Section~1.6]{AKOV-similarityclasses-A2}.

\subsection{Beyond $\GL_{N}$}

It is natural to ask whether it is possible to construct regular representations
of reductive groups over $\mfo_{r}$ other than $\GL_{N}$. As we
have already mentioned, \cite{KOS} constructs regular representations
for $\SL_{N}(\mfo_{r})$, $p\nmid N$, as well as for unitary groups.
These cases are relatively close to $\GL_{N}$, but one may expect
that it is possible to construct the regular representations of $G(\mfo_{r})$
whenever $G$ is a sufficiently nice reductive group scheme over $\mfo$,
for example when the derived group of $G$ is simply connected and
$p$ is a very good prime. The first step is to show that under some
hypotheses on $G$, any $\beta\in\Lie(G)(\mfo_{r})$ such that $\bar{\beta}\in\Lie(G)(\F_{q})$
is regular, will have abelian centraliser in $G(\mfo_{r})$ and the
surjective mapping property of centralisers under reduction maps.

\subsection{Beyond regular representations}

Hill's construction of strongly semisimple representations (see Section\ref{sec:Hill-Takase})
shows that Clifford theoretic methods can be used to construct some
non-regular representations of $\GL_{N}(\mfo_{r})$, up to knowledge
of all the irreducible representations of $\GL_{N'}(\mfo_{r'})$ for
$N'<N$, $r'<r$. Is there a uniform construction which includes the
regular representations and the strongly semisimple representations
(and perhaps others)?

\subsection{Relation with supercuspidal types}

Henniart \cite{Henniart-appendix} and Paskunas \cite{Vytas-unicity}
have shown that every supercuspidal representation of $\GL_{N}(F)$
has a unique \emph{type} on $\GL_{N}(\mfo)$. It would be interesting
to identify the regular representations which are supercuspidal types
and determine what they map to under the inertial Langlands correspondence.

\subsection{Onn's conjectures}

For each integer $n\geq1$, let 
\[
r_{n}=r_{n}(G_{r})=\#\{\pi\in\Irr(G_{r})\mid\dim\pi=n\}.
\]
The experience with the known cases of $\GL_{2}(\mfo_{r})$ \cite{Alex_smooth_reps_GL2,Uri-rank-2},
$\GL_{3}(\mfo_{r})$ \cite{AKOV-similarityclasses-A2} and the regular
representations of $\GL_{N}(\mfo_{r})$, suggests that $r_{n}$, as
a function of $\mfo_{r}$, is rather well behaved. More precisely,
in all known cases, it is a polynomial over $\Q$ in the size $q$
of the residue field, independent of the compact DVR $\mfo$, as long
as the residue field is $\F_{q}$. Moreover, the dimensions of the
known representations of $\GL_{N}(\mfo_{r})$ are given by polynomials
in $q$, and one may ask whether this is true in general. In \cite{Uri-rank-2}
Onn made the following conjectures, which we paraphrase slightly and
state only for $\GL_{N}(\mfo_{r})$:\\

\begin{conjecture*}
[Onn]~

\begin{enumerate}
\item Suppose $\mfo$ and $\mfo'$ are two compact DVRs with maximal ideals
$\mfp$ and $\mfp'$, respectively, such that $|\mfo/\mfp|=|\mfo'/\mfp'|$.
Then there is an isomorphism of group algebras 
\[
\C[\GL_{N}(\mfo_{r})]\cong\C[\GL_{N}(\mfo_{r}')].
\]
\item For any $n\geq1$ there exists a polynomial $p_{n}(x)\in\Q[x]$ such
that for any compact DVR $\mfo$ we have 
\[
r_{n}(\GL_{N}(\mfo_{r}))=p_{n}(q),
\]
where $q=|\mfo/\mfp|$.
\item There exist finitely many polynomials $d_{1}(x),\dots,d_{h}(x)\in\Z[x]$
with $\deg d_{i}\leq\binom{N}{2}r$, such that for any compact DVR
$\mfo$ we have 
\[
\{\dim\pi\mid\pi\in\Irr(\GL_{N}(\mfo_{r})),\ \pi\ \text{primitive}\}=\{d_{1}(q),\dots,d_{h}(q)\},
\]
where $q=|\mfo/\mfp|$.
\end{enumerate}
\end{conjecture*}
Note that part $(ii)$ of this conjecture implies part $(i)$.

\bibliographystyle{alex}
\bibliography{alex}

\end{document}